\documentclass[10pt]{amsart}
\usepackage{amssymb,amsfonts}
\usepackage{amsmath,amscd}
\usepackage[all]{xy}
\usepackage{enumitem}
\usepackage{mathrsfs}
\usepackage{ stmaryrd }
\usepackage{url}
 \usepackage[foot]{amsaddr}

\usepackage{hyperref}
\hypersetup{
    colorlinks=true,
    linkcolor=blue,
    urlcolor=blue, 
    citecolor= blue}
\usepackage[noabbrev, capitalize, nameinlink]{cleveref}

\newcommand\define{\newcommand}

\DeclareMathOperator{\coker}{coker}

\define\Div{\mathrm{Div}}

\define\isoto{\xrightarrow{\sim}}
\define\onto{\twoheadrightarrow}
\DeclareMathOperator{\Spec}{Spec}

\define\E{\mathcal{E}}

\newcommand{\dia}[1]{{\langle #1 \rangle}}

\newcommand{\Z}{\mathbb{Z}}
\newcommand{\Q}{\mathbb{Q}}
\newcommand{\C}{\mathbb{C}}

\newcommand{\F}{\mathbb{F}}

\newcommand{\sO}{\mathcal{O}}

\newcommand{\p}{\mathfrak{p}}
\newcommand{\m}{\mathfrak{m}}

\newcommand{\Hom}{\mathrm{Hom}}

\newcommand{\Ext}{\mathrm{Ext}}
\newcommand{\Res}{\mathrm{Res}}
\newcommand{\End}{\mathrm{End}}

\newcommand{\Fr}{\mathrm{Fr}}

\newcommand{\lb}{{[\![}}
\newcommand{\rb}{{]\!]}}

\newcommand{\tC}{\tilde{C}}

\define\Pic{{\mathrm{Pic}}}

\define\GL{{\mathrm{GL}}}

\define\kcyc{\kappa_{\mathrm{cyc}}}

\define{\Fitt}{\mathrm{Fitt}}
\define{\Ann}{\mathrm{Ann}}

\newtheorem{thm}{Theorem}[section] 
\newtheorem*{thm*}{Theorem}

\newtheorem{cor}[thm]{Corollary}
\newtheorem{prop}[thm]{Proposition}
\newtheorem{lem}[thm]{Lemma}

\theoremstyle{definition}

\theoremstyle{remark}
\newtheorem{rem}[thm]{Remark}

\newcommand{\bT}{\mathbb{T}}

\usepackage{color}

\newcommand{\tr}{{\mathrm{tr}}}

\newcommand{\sm}[4]{\ensuremath{\big(\begin{smallmatrix}#1 & #2 \\ #3 & #4\end{smallmatrix}\big)}}

\newcommand{\an}{\mathrm{an}}

\newcommand{\cD}{\mathcal{D}}

\newcommand{\Zp}{\Z_{(p)}}

\newcommand{\emrule}{\thinspace---\hskip.16667em\relax}

\newcommand{\tT}{\tilde{\bT}}
\newcommand{\tI}{\tilde{I}}

\makeatletter
\let\c@equation\c@thm
\makeatother
\numberwithin{equation}{section}

\title{	Another look at rational torsion of modular Jacobians}

\author{Kenneth A.~Ribet and Preston Wake}
\address[Ribet]{Department of Mathematics, University of California, Berkeley, CA 94720-3840}
\email{ribet@math.berkeley.edu}

\address[Wake]{Department of Mathematics, Michigan State University, 
East Lansing, MI 48824}
\email{wakepres@msu.edu}

\begin{document}
\begin{abstract}
We study the rational torsion subgroup of the modular Jacobian $J_0(N)$ for $N$ a square-free integer. We give a new proof of a result of Ohta on a generalization of Ogg's conjecture: for a prime number $p \nmid 6N$, the $p$-primary part of the rational torsion subgroup equals that of the cuspidal subgroup.  Whereas previous proofs of this result used explicit computations of the cardinalities of these groups, we instead use their structure as modules for the Hecke algebra. 
\end{abstract}
\maketitle

\section{Introduction}
Let $N$ be a square-free integer and let $J_0(N)$ be the Jacobian of the modular curve $X_0(N)$. 
In the case where $N$ is prime, Ogg computed the order of the cuspidal subgroup (the subgroup of $J_0(N)$ generated by linear equivalence classes of differences of cusps) \cite{ogg1973} and conjectured that the cuspidal subgroup is the whole rational torsion subgroup  \cite{ogg1975}. Mazur's proof of Ogg's conjecture was one of the principal results of \cite{mazur1978}.

Ogg's computation of the order of the cuspidal subgroup has been generalized to other modular curves by Kubert and Lang (see \cite{KL1981} for example).  
The order of the cuspidal subgroup for $X_0(N)$ with square-free $N$ has been computed by Takagi \cite{takagi1997} using Kubert--Lang theory.  

A number of authors have considered the generalization of Ogg's
conjecture to $J_0(N)$ where $N$ is a positive integer that is not
necessarily prime.  Since the cuspidal subgroup of $J_0(N)$ may not
consist entirely of rational points, the generalization states that
the rational torsion subgroup of $J_0(N)$ is \emph{contained in} the
cuspidal subgroup of this Jacobian.  See Lorenzini \cite{lorenzini1995}, Conrad--Edixhoven--Stein \cite{CES2003}, Ohta \cite{ohta2013,ohta2014}, Yoo \cite{yoo2016},  and Ren \cite{ren2018} for details.
In particular, Ohta \cite{ohta2014} has proven the generalization of Ogg's conjecture to square-free $N$. That is, he proved that the $p$-primary parts of
$J_0(N)(\Q)_\mathrm{tor}$ and of the cuspidal subgroup are equal for $p \ge 5$ (and $p=3$ if $3 \nmid N$).  The computation of the order of the class group by Takagi was a key input in Ohta's proof.  

In this article,  we give a new proof that, for a prime $p \nmid 6N$, the $p$-primary parts of $J_0(N)(\Q)_\mathrm{tor}$ and of the cuspidal subgroup are equal. Our proof is by ``pure thought"\emrule we do not compute the order of either group but instead analyze their structure as modules for the Hecke algebra. 

Because we do not consider the $p$-primary parts for primes $p$ dividing $6N$, our results do not completely recover Ohta's.  We have not investigated the possibility of adapting our
method to handle these primes.
See the PhD thesis of Lui \cite{lui2019} for some computational evidence for the generalized Ogg conjecture for $2$- and $3$-primary parts.

\subsection{Structure of the proof}
Fix a prime $p \nmid 6N$. Let $C$ denote the $p$-primary part of cuspidal subgroup and $T$ denote the $p$-primary part of the rational torsion subgroup $J_0(N)(\Q)_\mathrm{tor}$.  Our proof of the following theorem is based on the structure of $C$ and $T$ as Hecke modules.   See \cref{subsec:modular setup} for precise definitions of the Hecke operators we use.
\begin{thm}
\label{thm:C=T}
The rational torsion subgroup $T$ coincides with its subgroup $C$.
\end{thm}
Let $\bT$ be the Hecke algebra acting on $J_0(N)$ and let $I \subset \bT$ be the Eisenstein ideal (see \cref{subsec:hecke alg} for the precise definitions). The following result is the essential contribution of this article.

\begin{thm}
\label{thm:Anns}
The annihilators of $T$ and $C$ as $\bT$-modules satisfy 
\[
\Ann_\bT(C) \subseteq I \subseteq \Ann_\bT(T).
\]
\end{thm}
Because the annihilator of $T$ is contained in the annihilator of its submodule $C$, the two inclusions of \cref{thm:Anns} imply the equalities
\begin{equation}
\label{eq:anns}
\Ann_\bT(C)=I= \Ann_\bT(T).
\end{equation}
We deduce \cref{thm:C=T} from \eqref{eq:anns} at the end of \cref{subsec: Eichler-Shimura}, using an argument of Mazur to show that the Pontryagin dual of $T$ is cyclic as a $\bT$-module. 

\begin{cor}
\label{cor:order of T and C}
The index of $I$ in $\bT$ is the order of $C=T$.
\end{cor}
\begin{proof}
Indeed, the Pontryagin dual of $T$ is isomorphic to $\bT/I$ because it is a cyclic $\bT$-module by \cref{thm:T dual cyclic} and has annihilator $I$ by \eqref{eq:anns}.
\end{proof}

We prove the inclusions $I \subseteq \Ann_\bT(T)$ and $\Ann_\bT(C) \subseteq I$ separately, by independent arguments. 
The Eichler--Shimura relation implies $T$ is annihilated by $T_q-q-1$ for almost all primes $q$ (see \cref{lem:E-S rel}). So, to prove the inclusion $I \subseteq \Ann_\bT(T)$, it is enough to show that elements of this type generate $I$. That is the content of the following theorem, which we prove in \cref{sec:Galois reps} using Galois representations.
\begin{thm}
\label{prop:gens of I}
Let $S$ be a finite set of prime numbers that contains all primes dividing $6Np$.  The Eisenstein ideal $I$ is generated by the set of all $T_q-q-1$ with $q \not \in S$.
\end{thm}
Let $\tT$ be the Hecke algebra acting on the full space of modular forms (not just cusp forms) of weight $2$ and level $\Gamma_0(N)$. Let $\tI \subset \tT$ be the Eisenstein ideal and $J \subset \tT$ be the ideal generated by the set of all $T_q-q-1$ for all $q \notin S$.  
Let 
\[
\alpha: \tT/J \onto \tT/\tI
\]
be the quotient map arising from the inclusion $J \subseteq \tI$.
The content of \cref{prop:gens of I} is that $J = \tI$, or, in other words, that $\alpha$ is an isomorphism. To give the flavor of the proof of \cref{prop:gens of I}, we now explain informally why $\tT/J$ and $\tT/\tI$ have the same minimal primes: Think of $\Spec(\tT/\tI)$ as the set of normalized Eisenstein eigenforms and $\Spec(\tT/J)$ as the set of those normalized eigenforms whose associated Galois pseudorepresentation is $\epsilon+1$, where $\epsilon: G_\Q \to \Z_p^\times$ is the $p$-adic cyclotomic character. 
Because of oldforms, the map sending an normalized eigenform to its Galois pseudorepresentation is not necessarily injective. However, for each $\ell \mid N$, the $U_{\ell}$-eigenvalue singles out a root of the characteristic polynomial of  Frobenius at $\ell$, and the map sending an normalized eigenform to its pseudorepresentation plus this additional data \emph{is} injective. 
We show that the normalized Eisenstein eigenforms fill out all the possibilities for roots of the characteristic polynomial of Frobenius at $\ell$ with pseudorepresentation $\epsilon+1$.  This shows that $\Spec(\tT/\tI) \subseteq \Spec(\tT/J)$ is an equality.  In \cref{sec:Galois reps}, we give a more precise version of this argument to establish that $\alpha$ is an isomorphism.

Our proof of the inclusion $\Ann_\bT(C) \subseteq I$ is given in \cref{sec:Ann(C)}. It is based on an interpretation, due to Stevens \cite{stevens1985},  of $C$ in terms of the lattice of Eisenstein series that have \emph{integral periods}. We use an integrality result about modular units to show that this lattice is contained in the lattice of Eisenstein series with integral $q$-expansion and thereby exhibit a subquotient of $C$ that has annihilator~$I$.

\begin{rem}
An alternate proof of the inclusion $\Ann_\bT(C) \subseteq I$ will be given in forthcoming work of Jordan, Ribet, and Scholl related to \cite{JRS2022}. This proof uses $p$-adic Hodge theory to show that the Hecke algebra $\tT$ is $p$-saturated in the endomorphism ring of the generalized Jacobian of $X_0(N)$ relative to the cusps, and then proceeds along the lines of the proof of \cite[(1.7)]{ribet1983}.
\end{rem}

\begin{rem}
In \cref{sec:units}, we prove an integrality result about modular units  (\cref{prop:integrality of units}) using arithmetic geometry techniques.  This integrality also follows from an explicit description of the modular units given by Takagi \cite{takagi1997} using Kubert--Lang theory.  We avoid invoking Takagi's result both in accordance with our desire for a `pure thought' proof and to show that our method can be used to give alternate proofs of some of his results (see \cref{cor:lattices equal}).
\end{rem}

\section{Preliminaries}
In this section we set up notation for modular forms and Hecke algebras, recall a duality result, and employ arguments from \cite{mazur1978} used to reduce the proof of \cref{thm:C=T} to \cref{thm:Anns}.

\subsection{Modular forms}
\label{subsec:modular setup}Let $M_2(N)$ denote the $\C$-vector space of modular forms of weight 2 and level $\Gamma_0(N)$ and let $S_2(N)$ and $E_2(N)$ be the subspaces of $M_2(N)$ consisting of cusp forms and Eisenstein series, respectively.  For a subring $R \subset \C$, let $M_2(N,R) \subset M_2(N)$ denote the $R$-module of modular forms whose $q$-expansion (at the cusp $\infty$) is in $R\lb q \rb$. Let $S_2(N,R) = S_2(N) \cap M_2(N,R)$ and $E_2(N,R) =E_2(N) \cap M_2(N,R)$.

The dimension of $E_2(N)$ is $2^r-1$, where $r$ is the number of prime divisors of $N$ (see \cite[Theorem 4.6.2, pg.~133]{DS2005}, for example).   For each $d \mid N$ with $d >1$, let $E_d$ be the element of $E_2(N)$ such that for prime numbers $\ell$, 
\[
a_\ell(E_d) = \begin{cases}
1 & \mbox{if } \ell \mid d \\
\ell & \mbox{if }\ell \nmid d
\end{cases},
\]
if $\ell \mid N$, and $a_\ell(E_d) = 1+\ell$ if $\ell \nmid N$.  These eigenforms form a basis of $E_2(N)$. Since each $E_d$ is in $E_2(N,\Zp)$, we infer that $E_2(N,\Zp)$ is a free $\Zp$-module of rank $2^r-1$ (although not necessarily with $\{E_d\}$ as basis).

\subsection{Hecke operators}
\label{subsec:modular curves and Hecke operators}
The spaces of modular forms just introduced are equipped with actions of the classical Hecke operators $T_n$ for $n \ge 1$.  These operators arise from correspondences on modular curves (see, for example, \cite[Chapter 7]{shimura1971}).  As such, they also act on geometric objects such as $J_0(N)$ and the divisor group of cusps, but there is some ambiguity as to how a correspondence acts (using either the ``Picard" or the ``Albanese" functoriality\emrule see the discussion in \cite[pg.~443--444]{ribet1990}).  A summary of our conventions is:
\begin{itemize}
\item The endomorphism of $J_0(N)$ denoted $T_\ell$ here is the same as the one in \cite[pg.~444]{ribet1990} defined using Picard functoriality. It satisfies $T_\ell= \xi_\ell^*$,  where $\xi_\ell$ is the endomorphism defined by Shimura in \cite[Chapter 7]{shimura1971}.
\item The action of $T_\ell$ on the cusps is the \emph{dual} of the ``standard" action (for example, our action of $T_\ell$ on the cusps is the same as the the action of $^t T_\ell$ on the cusps described in \cite[pg.~15]{stevens1982}).
\end{itemize} 
To avoid any ambiguity, we now spell this out in more detail.

We use the following notation for modular curves: $Y_0(N)^\an$ is the quotient of the upper half-plane by $\Gamma_0(N)$, thought of as a complex analytic manifold, and $Y_0(N)$ is the smooth model of $Y_0(N)^\an$ over $\Z_{(p)}$ (recall that $p \nmid N$, so such a smooth model exists).   We use analogous notations $X_0(N)^\an$ and $X_0(N)$ for the closed modular curve.  Let $C_0(N) = X_0(N) - Y_0(N)$ denote the set of cusps. Let $J_0(N)$ denote the Jacobian variety of $X_0(N)$.

For a prime number $\ell$, let $t_\ell= \sm{1}{0}{0}{\ell} \in \GL_2(\Q)^+$ and consider the subgroup $\Gamma_0(N;\ell)= \Gamma_0(N)\cap t_\ell^{-1}  \Gamma_0(N)  t_\ell$.  Let $Y_0(N;\ell)^\an$ be the quotient of the upper half-plane by $\Gamma_0(N;\ell)$. There is a correspondence
\[
\xymatrix{
& Y_0(N;\ell)^\an \ar[dl]_-{\pi} \ar[dr]^-{t_\ell} & \\
Y_0(N)^\an & & Y_0(N)^\an
}
\]
where $\pi$ is the quotient map and $t_\ell$ is the map induced by $x \mapsto t_\ell x$.  This correspondence extends uniquely to give a correspondence on $X_0(N)$ that preserves the cusps; we use the same names $\pi$ and $t_\ell$ for the maps in this correspondence.  Define the element $T_\ell \in \End(J_0(N))$ by
\[
T_\ell = (t_\ell)_* \pi^*.
\]
The same formula defines an endomorphism $T_\ell$ of the group $\Div^0(C_0(N))$ of degree-zero divisors on the cusps.  The map
\[
\Div^0(C_0(N)) \to J_0(N)
\]
sending a divisor to its class is equivariant for these actions of $T_\ell$. 

By identifying $M_2(N)$ with the space of differential forms on $Y_0(N)^\an$, the formula $T_\ell =  (t_\ell)_* \pi^*$ gives an action of $T_\ell$ on $M_2(N)$.  This action preserves the subspaces $S_2(N)$ and $E_2(N)$, and preserves the $R$-submodule $M_2(N,R)$ if $R$ is a $\Z[1/N]$-algebra.  
\subsection{Hecke algebras}
\label{subsec:hecke alg}
 Let $\tT \subset \End(M_2(N))$ be the $\Zp$-algebra generated by the Hecke operators $T_\ell$ for all primes $\ell \ge 1$.  Let $\tI \subset \tT$ be the annihilator of the space of Eisenstein series. Let $\bT$ be the image of $\tT$ in $\End(S_2(N))$ and let $I \subset \bT$ be the image of $\tI$.  As is customary, for a prime divisor $\ell$ of $N$, we denote the operator $T_\ell$ by $U_\ell$.

\subsection{Duality}
The following duality is well known. The duality result for cuspforms (see, for example, \cite[Theorem 2.2]{ribet1983}) extends to all modular forms because there is no nonzero mod $p$ modular form of weight $2$ and level $\Gamma_0(N)$ with constant $q$-expansion when $p>3$.  This fact can be proven using a mod $p$ Atkin--Lehner-type result (see \cite[Lemma II.5.9, pg.~83]{mazur1978} for the prime-level version; the proof generalizes to higher levels as in \cite[Lemma 3.5]{agashe2018} or \cite[Lemma (2.1.1)]{ohta2014}) together with the fact that there are no nonzero mod $p$ modular forms of weight $2$ and level $1$ \cite[Proposition II.5.6, pg.~81]{mazur1978}.
Alternatively,  a nonzero constant is the $q$-expansion of a modular form of weight $0$, but,  by \cite[Some Corollaries (2), pg.~55]{katz1977}, a nonzero mod $p$ modular form of weight $2$ cannot have the same $q$-expansion as a modular form of smaller weight.
\begin{lem}
\label{lem:duality} 
Let $M= M_2(N,\Zp)$. 
\begin{enumerate}
\item The pairing
\begin{equation}
\label{eq:pair M and T}
M \times \tT \to \Zp, \ (f,T) \mapsto a_1(Tf)
\end{equation}
is a perfect pairing of free $\Zp$-modules of finite rank. 
\item Let $X \subset M$ be a $\tT$-submodule.  The pairing \eqref{eq:pair M and T} induces an isomorphism
\[
\Hom(M/X,\Zp) \cong \Ann_{\tT}(X).
\]
Moreover, if $M/X$ is torsion free, then \eqref{eq:pair M and T} induces an isomorphism 
\[
\Hom(X,\Zp) \cong \tT/\Ann_{\tT}(X).
\]
\end{enumerate}

\end{lem}
\begin{proof} 
The analogue of (1) with $\Q$-coefficients is easy and standard. This analogue implies that the map
\[
\Psi: M_2(N,\Zp) \to \Hom_{\Zp}(\tT,\Zp),  \ \Psi(f)(T)=a_1(Tf)
\]
is injective and that $\coker\Psi$ is a finite abelian $p$-group.  To show that $\Psi$ is an isomorphism, it suffices to show that $\coker\Psi$ is $p$-torsion free. Let  $\phi \in \Hom_{\Zp}(\tT,\Zp)$ be such that $p\phi= \Psi(f)$ for some $f \in M_2(N,\Zp)$; we will show that $\phi$ is in the image of $\Psi$, and hence that $\coker\Psi$ is $p$-torsion free. 
Note that for all $n \ge 1$,
\[
a_n(f)=\Psi(f)(T_n)=p\phi(T_n).
\]
Since $\phi(T_n) \in \Zp$ for all $n \ge1$, this implies that $f \pmod p$ is a constant and thus it is $0$, as was recalled in the sentences before the statement of the lemma. Hence $g:=p^{-1}f $ is in  $M_2(N,\Zp)$ and $\phi=\Psi(g)$ is in the image of $\Psi$.  This shows that $\Psi$ is an isomorphism, completing the proof of (1).  Statement (2) follows immediately from (1).
\end{proof}

\begin{lem}
\label{lem:rank of T/I}
The $\Zp$-module $\tT/\tI$ is free of rank $2^r-1$, where $r$ is the number of prime divisors of $N$.
\end{lem}
\begin{proof}
\cref{lem:duality}(2) applied to $X=E_2(N,\Zp)$ gives an isomorphism
\[\tT/\tI \cong \Hom_{\Zp}(E_2(N,\Zp),\Zp).\]
In particular,  $\tT/\tI$ is a free $\Zp$-module of the same rank as $E_2(N,\Zp)$, which is $2^r-1$ as discussed in \cref{subsec:modular setup}.
\end{proof}

\subsection{The Eichler--Shimura relation}
The following lemma, which is well known, shows that many of the elements of $I$ annihilate $T$.  In \cref{sec:Galois reps}, we will prove the inclusion $I \subseteq \Ann_\bT(T)$ by showing that these elements generate $I$.
\label{subsec: Eichler-Shimura}
\begin{lem}
\label{lem:E-S rel}
For every prime $q \not \in S$, the group $T$ is annihilated by $T_q-q-1$. 
\end{lem}
\begin{proof}
Let $q$ be a prime that is not in $S$. The Eichler--Shimura relation states that $T_q = \Fr_q + V_q$ on $J_0(N)_{/\F_q}$ \cite[pg.~89]{mazur1978}, where $\Fr_q$ is the Frobenius endomorphism of the group scheme $J_0(N)_{/\F_q}$ and $V_q$ is the Verschiebung. Since $\Fr_q$ is the identity on $J_0(N)(\F_q)$,  this implies that $T_q=1+q$ on $J_0(N)(\F_q)$. But, since $q \nmid 2N$, the reduction modulo $q$ map induces an injection $T \to J_0(N)(\F_q)$ of $\bT$-modules (see \cite[Appendix]{katz1981}, for example), so $T_q = 1 +q$ on $T$ as well.
\end{proof}

\subsection{Cyclicity of the dual of $T$}
\cref{thm:C=T} follows from \cref{thm:Anns} together with the following mild generalization of a theorem Mazur \cite[Corollary II.14.8 pg.~199]{mazur1978}, which appears in the work of Ohta \cite[Proposition (3.5.4)]{ohta2014}.

\begin{thm}[Mazur, Ohta]
\label{thm:T dual cyclic}
The Pontryagin dual of $T$ is cyclic as a $\bT$-module.
\end{thm}
\begin{proof}
It is equivalent to show that,  for all maximal ideals $\m$ of $\bT$,  the $\m$-torsion subgroup of $T$, which we denote by $T[\m]$, has dimension at most $1$ as a $\bT/\m$-vector space.  
Since, as was noted in the proof of \cref{lem:E-S rel}, the reduction map $T \to J_0(N)(\F_p)$ is injective, $T[\m]$ is contained in the $\m$-torsion subgroup of $J_0(N)[p] (\bar{\F}_p)$.  Furthermore,  as in \cite[Proposition II.14.7 pg.~119]{mazur1978}, the Cartier--Serre isomorphism induces an injective homomorphism
\[
J_0(N)[p] (\bar{\F}_p)\otimes_{\F_p} \bar{\F}_p \ \hookrightarrow \ H^0(X_0(N)_{/\bar{\F}_p},\Omega^1).
\] 
Note that, as in Mazur's proof, the Cartier--Serre isomorphism is Hecke-equivariant with respect to the our chosen Hecke action on $J_0(N)$ (see also \cite[Proposition 6.5]{wiles1980}), so this is a homomorphism of $\bT$-modules.
The $\m$-torsion subgroup of the right-hand side has dimension at most $1$ as a $\bT/\m$-vector space by the $q$-expansion principle.  Hence $T[\m]$ has dimension at most $1$ as a $\bT/\m$-vector space.
\end{proof}

\subsection{Equality of annihilators implies equality} The cyclicity proven in \cref{thm:T dual cyclic} implies that the equality $T=C$ of $\bT$-modules follows from the equality $\Ann_{\bT}(T)=\Ann_{\bT}(C)$ of annilhilators:

\begin{proof}[Proof of \cref{thm:C=T} assuming \cref{thm:Anns}]
The inclusion $C \subseteq T$ induces a surjection $T^\vee \onto C^\vee$ of Pontryagin duals. By \cref{thm:Anns} and \eqref{eq:anns}, this surjection is a map of $\bT/I$-modules and $C^\vee$ is a faithful $\bT/I$-module.  
Then the map $T^\vee \onto C^\vee$ is an isomorphism because, 
by \cref{thm:T dual cyclic},  $T^\vee$ is a cyclic $\bT$-module.
\end{proof}

\section{Galois representations and the annihilator of $T$}
\label{sec:Galois reps}
In this section, we prove \cref{prop:gens of I} and the inclusion $\Ann_{\bT}(T) \supseteq I$ of \cref{thm:Anns}.
Fix a finite set $S$ of prime numbers containing all primes dividing $6Np$.

\subsection{Galois representations}
\label{subsec:galois reps}
Let $J \subset \tT$ be the ideal generated by $T_q-q-1$ for all $q \not \in S$.
\begin{lem}
\label{lem:J contains}
The ideal $J$ contains the following elements:
\begin{enumerate}[label=(\alph*)]
\item $T_q-q-1 $ for each prime $q \nmid N$, 
\item $(U_\ell-1)(U_\ell-\ell)$ for each prime $\ell \mid N$,  and
\item $\prod_{\ell \mid N} (U_\ell-1)$, where the product is over the set of prime divisors of $N$.
\end{enumerate}
\end{lem} 

We will prove \cref{lem:J contains} using properties of Galois representations associated to eigenforms.  Before continuing with the proof, we review the required properties. To this end,  fix a maximal ideal $\m \subset \tT$ and assume that $\m \supseteq J$.

For each minimal prime ideal $\p \subset \m$, there is a corresponding (not necessarily cuspidal) eigenform $f_\p$ with coefficients in $\sO_\p=\tT_\m / \p$ (which is a finite extension of $\Z_p)$, defined by $a_n(f_\p) = T_n \pmod{\p}$.  Let $K_\p$ be the field of fractions of $\sO_\p$. 
To each $f_\p$, there is an associated two-dimensional semisimple Galois representation $\rho_\p$ with coefficients in $K_\p$, contructed by Shimura. See \cite[Theorem 3.1]{DDT1997} for a list the properties of $\rho_\p$, which have been established by the work of many mathematicians; we will rewrite these properties in terms of a representation $\rho_\m$ that we now define.

Since $\bT_\m$ is reduced (see \cite{CE1998}),  there is an injective homomorphism
\begin{equation}
\label{eq:normalization of tT}
\tT_\m \hookrightarrow \tT_\m \otimes \Q = \prod_{\p \mid \m} K_\p.
\end{equation}
The product of the Galois representations $\rho_\p$ is a continuous representation
\[
\rho_\m: G_\Q \to \GL_2(\tT_\m \otimes \Q)
\]
with the following properties:
\begin{enumerate}
\item $\rho_\m$ is unramified outside $Np$,
\item $\det(\rho_\m) = \kcyc$, the $p$-adic cyclotomic character,
\item $\tr(\rho_\m)(\Fr_q) = T_q$ for each prime $q \nmid Np$, where $\Fr_q$ is an arithmetic Frobenius.
\end{enumerate}
Since $\m \supseteq J$,  the Chebotarev density theorem implies that $\tr(\rho_\m)(\sigma) \equiv \kcyc(\sigma)+1 \pmod{\m}$ for every $\sigma \in G_\Q$. This implies that $\m$ is ordinary in the sense that $T_p$ is invertible modulo $\m$. Indeed,  a theorem of Fontaine (see \cite[Theorem 2.6]{edixhoven1992}) implies that the restriction of $\tr(\rho_\m) \pmod{\m}$ to the inertia group at $p$ is the sum of two non-trivial characters in the non-ordinary case. Let $u_p \in \tT_\m^\times$ denote the unit root of $x^2-T_px+p$, which exists by Hensel's lemma.

The restrictions of $\rho_\m$ to decomposition groups at primes dividing $Np$ can be described as follows:
\begin{enumerate}[resume]
\item For every prime $\ell \nmid N$ and every choice of arithmetic Frobenius $\Fr_\ell \in G_\Q$ the trace of $\rho_\m(\Fr_\ell)$ is given by
\[
\tr(\rho_\m)(\Fr_\ell) = U_\ell + \ell U_\ell^{-1}.
\]
\item For every $\sigma$ in a decomposition group at $p$,  the trace of $\rho_\m(\sigma)$ is given by
\[
\tr(\rho_\m)(\sigma) = \kcyc(\sigma)\lambda^{-1}(\sigma) +\lambda(\sigma)
\]
where $\lambda$ is the unramified character sending Frobenius to $u_p$
\end{enumerate}

\begin{proof}[Proof of \cref{lem:J contains}]
It is enough to show the containment after completion at all maximal ideals $\m \subset \tT$.  We can and do assume $\m \supseteq J$ because the statement is clear if $J_\m = \tT_\m$.

The Chebotarev density implies that the image of $\tr(\rho_\m)$ is contained in $\tT_\m$ and that $J_\m$ is the ideal generated by
\[
\tr(\rho_\m)(\sigma) - \kcyc(\sigma) -1
\]
for all $\sigma \in G_\Q$. It follows from (3) that $T_q -q-1 \in J_\m$ for all $q \nmid N$.  

To prove part $(a)$, it remains to show that $T_p -p-1 \in J_\m$, or, equivalently, that $u_p -1 \in J_\m$. 
Let $\sigma$ be an element of the inertia group at $p$ such that $\kcyc(\sigma) \not\equiv 1 \bmod{p}$. Then (5) implies
\begin{align*}
\tr(\rho_\m)(\Fr_p) - \kcyc(\Fr_p) -1  & = \kcyc(\Fr_p)(u_p^{-1}-1) +u_p-1 \in J_\m, \\
\tr(\rho_\m)(\sigma\Fr_p) - \kcyc(\sigma\Fr_p) -1  & =\kcyc(\sigma) \kcyc(\Fr_p)(u_p^{-1}-1) +u_p-1 \in J_\m.
\end{align*}
Subtracting these equations, we find that $(\kcyc(\sigma)-1) \kcyc(\Fr_p)(u_p^{-1}-1) \in J_\m$.  By the assumption on $\sigma$, the element $(\kcyc(\sigma)-1)\kcyc(\Fr_p)\in \Z_p$ is a unit, so $(u_p^{-1}-1) \in J_\m$. Hence $T_p-p-1 \in J_\m$, completing the proof of $(a)$.

For part $(b)$, note that by (4)
\[
\tr(\rho_\m)(\Fr_\ell)-\ell -1 = U_\ell +\ell U_\ell^{-1} -\ell -1 \in J_\m.
\]
Since $(U_\ell-1)(U_\ell-\ell) = U_\ell ( U_\ell +\ell U_\ell^{-1} -\ell -1)$, this implies 
 $(U_\ell-1)(U_\ell-\ell) \in J_\m$.

For part $(c)$,  let $x = \prod_{\ell \mid N} (U_\ell-1)$. We will show that $x=0$ in $\tT_\m$ by showing that $x$ maps to zero in each factor $K_\p$ under the injective map \eqref{eq:normalization of tT}.  
First suppose that $f_\p$ is an Eisenstein series. 
The only eigenforms in $E_2(N)$ are the forms $E_d$ defined in \cref{subsec:modular setup}, hence $f_\p=E_d$ for some $d \mid N$ with $d>1$. In particular, $a_\ell(f_\p)=1$ for every prime $\ell \mid d$, so $x$ maps to zero in that factor.  Next suppose that $f_\p$ is a cuspform. Then $a_\ell(f_\p)=1$ for some $\ell \mid N$ by the following lemma of Ribet and Yoo.  Hence $x=0$ in $\tT_\m$. 
\end{proof}

\begin{lem}[Ribet, Yoo]
Let $M$ be a square-free integer and let $p \nmid 6M$. Suppose that $f$ is a newform of weight 2 and level $M$ such that $a_q(f) \equiv q +1 \pmod{\p}$ for all $q \nmid Mp$. Then $a_\ell(f) =1$ for some $\ell \mid M$. 
\end{lem}
\begin{proof}
See \cite{yoo2017}. 
\end{proof}

\subsection{Proof that $J=\tI$}
We now prove \cref{prop:gens of I}, which is the claim that $J=\tI$, by showing that the elements of $J$ listed in \cref{lem:J contains} generate $\tI$.
 
\begin{proof}[Proof of \cref{prop:gens of I}]
We will show that the natural surjection $\alpha: \tT/J \onto \tT/\tI$ is an isomorphism. This will imply that the elements $T_q-q-1$ for $q \not\in S$ generate $\tI$ and hence also generate $I$.

Let $N=\ell_1\cdots \ell_r$ be the prime factorization of $N$ and let $R=\Zp[x_1,\dots,x_r]$.
By \cref{lem:J contains} part (a), the $\Zp$-algebra homomorphism
\[
s: R \to \tT/J, \ x_i \mapsto U_i-1
\]
is surjective. Let $\mathcal{I} \subset R$ be the ideal generated by $x_1x_2\cdots x_r$ and $x_i(x_i+1-\ell_i)$ for $i=1,\dots,r$.  
The map $s$ induces a surjection $\bar{s}:R/\mathcal{I} \onto \tT/J$ because $s(\mathcal{I})=0$ by \cref{lem:J contains} parts (b) and (c).  We claim that $R/\mathcal{I}$ is a free $\Zp$-module of rank $2^r-1$, which is the same as the rank of $\tT/\tI$ by \cref{lem:rank of T/I}. Indeed,  consider the ideal $\mathcal{I}' \subset \mathcal{I}$ generated by $x_i(x_i+1-\ell_i)$ for $i=1,\dots,r$. Then $R/\mathcal{I}'$ is freely generated by the $2^r$ monomials of degree at most $1$ in each $x_i$.  The ideal generated by $x_1x_2\cdots x_r$ in $R/\mathcal{I}'$ is equal to its $\Zp$-span, which is a direct summand of $R/\mathcal{I}'$ as $\Zp$-modules.  Since $\mathcal{I}=\mathcal{I}'+x_1x_2\cdots x_rR$, this implies that $R/\mathcal{I}$ is free of rank $2^r-1$.

The composition
\[
\alpha \circ \bar{s}: R/\mathcal{I} \onto \tT/\tI
\]
is a surjective homomorphism of free $\Zp$-modules of the same finite rank and is therefore an isomorphism. Since $\bar{s}$ is surjective, this implies that $\alpha$ is an isomorphism,  which is the content of the theorem.
\end{proof}

Since $\Ann_{\bT}(T) \supseteq J$ by \cref{lem:E-S rel}, \cref{prop:gens of I} implies the inclusion $\Ann_{\bT}(T) \supseteq~I$, which is one of the two inclusions of \cref{thm:Anns}.

\section{The period lattice of Eisenstein series}
\label{sec:units}
In this section,  we introduce the period lattice in the space of Eisenstein series. This lattice is closely related to the cuspidal subgroup of $J_0(N)$ and to modular units by the work of Stevens \cite{stevens1985}.  We show that the $q$-expansions of elements of the period lattice have coefficients in $\Z_{(p)}$ by first proving an analogous integrality property for $q$-expansions of modular units. 

Recall from \cref{subsec:modular curves and Hecke operators} the notation $Y_0(N)^\an \subset X_0(N)^\an$ for the analytic modular curves,  $Y_0(N) \subset X_0(N)$ for their $\Zp$-models, and $C_0(N)=X_0(N)-Y_0(N)$ for the set of cusps. If $R$ is a commutative $\Z_{(p)}$-algebra,  let $Y_0(N)_R$ and $X_0(N)_R$ denote the base-changes of $Y_0(N)$ and $X_0(N)$ to $R$. 
\subsection{The period lattice}
For each Eisenstein series $f \in E_2(N)$,  consider the period map $
\int f: H_1(Y_0(N)^\an,\Z) \to \C,$ defined by
\[
\int f:  \gamma  \mapstochar\longrightarrow \int_\gamma f(z)\,dz.
\]
The \emph{period lattice} $\E$ is the $\Zp$-module consisting of those Eisenstein series $f$ for which the image of $\int f$ is contained in $\Zp$. 

Let $U = (\sO_{Y_0(N)^\an}^\times /\C^\times) \otimes_\Z \Zp$ be the $\Zp$-module of (analytic) modular units modulo scalars.  We will refer to elements of $U$ as \emph{scalar classes} of modular units; note that the divisor of a modular unit depends only on its scalar class. Let $\tC = \Div^0(C_0(N)) \otimes_\Z \Zp$ be the $\Zp$-module of degree-zero divisors with cuspidal support. By definition,
\[
C = \tC /\mathrm{div}(U).
\]
Stevens establishes the isomorphism
\begin{equation}
\label{eq:D}
\cD:U \isoto \E,  \quad \cD(u)\,d\log(q)=d\log(u),
\end{equation}
where $d\log(q)=2 \pi i\,dz$. He proves that $\mathrm{div}(u)=\Res(2\pi i \cD(u))$ for all $u \in U$ \cite[pg.~521]{stevens1985}, where $\Res\colon\E~\to~\tC$ is the residue map defined by 
\[
\Res(f) =\sum_{x \in C_0(N)} \Res_x(f(z)dz) [x]. 
\]
 We will use this isomorphism, together with a integrality result for modular units, to study the integrality of $q$-expansions of elements of the period lattice.

\subsection{$\Zp$-integrality of modular units}
We show that every (analytic) modular unit of level $\Gamma_0(N)$ can be written as a complex constant times an \emph{$\Zp$-integral modular unit}\emrule that is, a unit in the ring of regular functions of the smooth model of the modular curve over $\Zp$.  Then we show that $\Zp$-integral modular units have $q$-expansions in $\Zp$ using the algebraic description of the infinity cusp.

The following lemmas provide the reduction from analytic modular units to integral modular units in two steps: first from analytic to rational, then from rational to integral. The second lemma is established using a similar method to that used in the proof of the $q$-expansion principle.

\begin{lem}
\label{lem:analytic to rational units}
The homomorphism
\[
\sO_{Y_0(N)_\Q}^\times \to \sO_{Y_0(N)^\an}^\times /\C^\times
\]
sending a unit on $Y_0(N)_\Q$ to the scalar class of its associated analytic modular unit,  is surjective. 
\end{lem}
\begin{proof}
Since the divisor map $\mathrm{div}: \sO_{Y_0(N)^\an}^\times /\C^\times \to  \Div^0(C_0(N))$ is injective,  the lemma states that for each analytic unit $u^\an \in \sO_{Y_0(N)^\an}^\times$, there is an algebraic unit $u \in \sO_{Y_0(N)_\Q}^\times$ with the same divisor.

Let $u^\an \in \sO_{Y_0(N)^\an}^\times$ and let $D =\mathrm{div}(u^\an)$ in $\tC$, and consider the divisor class $\bar{D} \in \Pic(X_0(N)_\Q)$ of $D$. Then $\bar{D}$ is in the kernel of the composite map
\[
\Pic(X_0(N)_\Q) \to \Pic(X_0(N)_\C) \to \Pic(X_0(N)^\an).
\]
The first map is injective by \cite[\href{https://stacks.math.columbia.edu/tag/0CC5}{Tag 0CC5}]{stacks-project}, and the second map is injective (in fact, an isomorphism) by GAGA (see \cite[Th\'eor\`eme 4.4, pg.~329]{GAGA}). This injectivity implies that $\bar{D}=0$ in $\Pic(X_0(N)_\Q)$, so $D=\mathrm{div}(u)$ for some rational function $u$ on $X_0(N)_\Q$.  Since $D$ is supported on the cusps,  $u \in \sO_{Y_0(N)_\Q}^\times$ as desired.
\end{proof}

\begin{lem}
\label{lem:rational to integral units}
Base change induces an surjection
\[
\sO_{Y_0(N)}^\times \to \sO_{Y_0(N)_\Q}^\times/\Q^\times.
\]
\end{lem}
\begin{proof}
Let $u \in  \sO_{Y_0(N)_\Q}^\times$, thought of as a rational function on the arithmetic surface $Y_0(N)$ with zeros and poles along the special fiber $Y_0(N)_{\F_p}$.  Choose a closed point $x \in Y_0(N)_{\F_p}$ with residue field $k$, a finite extension of $\F_p$, and identify the the local ring at $x$ with $W(k)\lb T \rb$, which is possible because $Y_0(N)$ is smooth. Consider the pullback $\tilde{u}$ of $u$ to $W(k)\lb T \rb$. Since $u$ is a unit on $Y_0(N)_\Q$,  the constant term $\tilde{u}(0)$ of $\tilde{u}$ is non-zero, so we can write $\tilde{u}(0)=p^n \alpha$ for some $n \in \Z$ and $\alpha \in W(k)^\times$.  We claim that $p^{-n}u$ is a unit in $\sO_{Y_0(N)}$.  
Indeed, the divisor of $p^{-n}u$ is supported on $Y_0(N)_{\F_p}$ by assumption, and, since $Y_0(N)$ is normal,  it must be a union of irreducible codimension $1$ subvarieties. Since $Y_0(N)$ is smooth, $Y_0(N)_{\F_p}$ itself is irreducible, so the divisor of $p^{-n}u$ is either empty or all of $Y_0(N)_{\F_p}$. 
But, by the definition of $n$, $x$ is not in the divisor of $p^{-n}u$, so the divisor is empty and $p^{-n}u$ is a unit.
\end{proof}

\begin{prop}
\label{prop:integrality of units}
The $q$-expansion of every element of $\sO_{Y_0(N)}^\times$, thought of as an analytic modular unit, is in $\Zp \lb q \rb[q^{-1}]^\times$. 
\end{prop}
\begin{proof}
The $q$-expansion of a modular unit is given by pulling back to the infinity cusp. The infinity cusp extends to the integral model as a morphism
\[
\infty: \Spec(\Z_{(p)}\lb q \rb [q^{-1}]) \to Y_0(N)
\]
of schemes (defined using the Tate elliptic curve). Hence there is a commutative diagram
\[
\xymatrix{
\sO_{Y_0(N)}^\times \ar[d] \ar[r] & \Z_{(p)}\lb q \rb [q^{-1}]^\times \ar[d] \\
\sO_{Y_0(N)^\an}^\times \ar[r] &  \C\lb q \rb [q^{-1}]^\times,
}
\]
where the horizontal arrows are the $q$-expansions (that is, pullback along $\infty$). 
This commutativity implies that the $q$-expansion of every element in the image of $\sO_{Y_0(N)}^\times \to \sO_{Y_0(N)^\an}^\times$ belongs to $\Z_{(p)}\lb q \rb [q^{-1}]^\times$.
\end{proof}

\subsection{$q$-expansions of elements of the period lattice}
We note the effect of the isomorphism $\cD$ from \eqref{eq:D} on $q$-expansions.  Let $u \in \sO_{Y_0(N)^\an}^\times$ and let $u(q)$ be the $q$-expansion of $u$ and $\cD(u)(q)$ be the $q$-expansion of $\cD(u)$. The chain rule and the fact that $dq = 2\pi i q\,dz$ imply 
\begin{equation}
\label{eq:q-exp of D}
\cD(u)(q) = q \frac{u'(q)}{u(q)},
\end{equation}
where $u'(q)=\frac{d}{dq}u(q)$.

\begin{thm}
\label{lem: period lattice in q lattice}
For each $f \in \E$, the $q$-expansion of $f$ has coefficients in $\Zp$. In other words, $\E \subseteq E_2(N,\Zp)$.
\end{thm}

\begin{proof}
Let $S \subset \sO_{Y_0(N)^\an}^\times$ be the image of the map $\sO_{Y_0(N)}^\times \to \sO_{Y_0(N)^\an}^\times$ sending an integral modular unit to the associated analytic modular unit.  By \cref{lem:analytic to rational units} and \cref{lem:rational to integral units}, the composition $S \to   \sO_{Y_0(N)^\an}^\times \to  \sO_{Y_0(N)^\an}^\times/\C^\times$ is surjective.  Then, by Steven's theorem  \eqref{eq:D},  $\cD(S)$ spans $\E$, so it is enough to show that $\cD(S) \subset~E_2(N,\Zp)$.  

Let $u \in S$ and write the $q$-expansion of $u$ as $q^n v$ for some $v\in \Zp \lb q \rb^\times$ and $n \in \Z$. 
By \eqref{eq:q-exp of D}, the $q$-expansion of $\cD(u)$ is $n+ \frac{ q v'}{v}$. Since $v'$ is in $\Z_{(p)} \lb q \rb$, the $q$-expansion of $\cD(u)$ is in $\Z_{(p)} \lb q \rb$ and $\cD(u) \in E_2(N,\Zp)$.
\end{proof}

\section{The annihilator of the cuspidal subgroup}
\label{sec:Ann(C)}

We established the inclusion $\Ann_\bT(T) \supseteq I$ in \cref{sec:Galois reps}.  To complete the proof
of \cref{thm:Anns}.  we need to show that $C$ is large in the sense that
$\Ann_\bT(C)$ is contained in $I$.  The following result implies this needed fact.

\begin{prop}
\label{cor:subquo of C}
There is a subquotient of $C$ whose Pontryagin dual is a free $\bT/I$-module of rank 1. 
\end{prop}
\begin{proof}
For this proof,  let $M=M_2(N,\Zp)$, $S=S_2(N,\Zp)$, and $E=E_2(N,\Zp)$.  
Stevens's isomorphism $\cD: U \isoto \E$ \eqref{eq:D} and the inclusion $\E \subset E$ of \cref{lem: period lattice in q lattice} induce a natural surjection 
\[
C \cong \tC / \Res(\E) \onto \tC/ \Res(E).
\]
On the other hand,  the residue theorem gives an exact sequence
\[
0 \to S \to M \xrightarrow{\Res} \tC,
\]
so the residue map induces an injective homomorphism
\[
\frac{M}{S+E} \to \frac{\tC}{\Res(E)}.
\]
Note that the residue map is Hecke-equivariant with respect to the action of Hecke operators on $\tC$ by Picard functoriality (see \cite[pg.~36]{stevens1982}, for example), which induces our chosen action of $\bT$ on $C$. 
Hence $C$ has a subquotient that is isomorphic to $\frac{M}{S+E}$. Let $X=\frac{M}{S+E}$; it remains to show that the Pontryagin dual $X^\vee$ of $X$ is isomorphic to $\bT/I$. 

Consider the exact sequence
\begin{equation}
\label{eq:res mod E}
0 \to S \to M/E \to X \to 0.
\end{equation}
By \cref{lem:duality}(2), there are duality isomorphisms 
\[
\Hom_{\Zp}(S,\Zp) \cong \bT,  \qquad \Hom_{\Zp}(M/E,\Zp) \cong \tI.
\]
Since $M/E$ is torsion free, $\Ext^1_{\Zp}(M/E,\Zp)=0$. Because $X$ is finite,  there is are isomorphisms $\Hom_{\Zp}(X,\Zp)=0$ and $\Ext^1_{\Zp}(X,\Zp) \cong X^\vee$.
Applying $\Hom_{\Zp}(-,\Zp)$ to \eqref{eq:res mod E}, we obtain an exact sequence
\[
0 \to \tI \to \bT \to X^\vee\to 0,
\]
so $X^\vee \cong \bT/I$.
\end{proof}

The proposition implies that $C$ has a subquotient whose annihilator is $I$. Since any element of $\bT$ that annihilates $C$ will annihilate this subquotient, it follows that $\Ann_\bT(C) \subseteq I$ and this completes the proof of \cref{thm:Anns}. By \cref{cor:order of T and C},  $C$ has the same cardinality as its subquotient $X$. We record some implications of this as a corollary.

\begin{cor} \hfill
\label{cor: gens of C dual}
\begin{enumerate}
\item The period lattice $\E$ coincides with the $q$-expansion lattice $E_2(N,\Zp)$.
\item There is an isomorphism of $\bT$-modules
\[
\frac{M}{S+E} \to C
\]
sending $f \in M$ to the class of $\Res(f)$ in $C$.
\item The Pontryagin dual of $C$ is a cyclic $\bT$-module generated by the homomorphism $\lambda: C \to \Q_p/\Z_p$ defined as follows: let $x \in C$,  choose $f \in M$ such that $\Res(f) \in \tC$ represents the class of $x$,  and write $f=a e + b g$ for $a, b \in \Q_p$, $e \in E$, and $g \in S$; then define $\lambda(x) = b a_1(g) \bmod{\Z_p}$.
\end{enumerate}
\end{cor}
\begin{proof}
Continue with the notation as in the proof of \cref{cor:subquo of C}.  Recall that $X = \frac{M}{S+E}$ is Pontriagyn dual to $\bT/I$, so there is an equality of cardinalities $|X|~=~|\bT/I|$.
The module $X$ is a subquotient of $C$ as follows
\begin{equation}
\label{eq:X subquo C}
\xymatrix{
X \ar@{^(->}[r]^-{\Res} & \tilde{C}/\Res(E) & \ar@{->>}[l]   \tilde{C}/\Res(\E)  \ar[r]^-\sim & C.
}
\end{equation}
Moreover,  $|C|=|\bT/I|$ by \cref{cor:order of T and C}.  The maps in \eqref{eq:X subquo C} imply a string of inequalities of cardinalities
\[
|\bT/I|= | X| \le  \left|\tilde{C}/\Res(E)\right| \le \left| \tilde{C}/\Res(\E)\right|  = |C| = |\bT/I|.
\]
Since the beginning and end of the string are equal,  all the inequalities are in fact equalities. This implies that all the maps in \eqref{eq:X subquo C} are isomorphisms.

The fact that the surjection $ \tilde{C}/\Res(\E) \onto \tilde{C}/\Res(E)$
is an isomorphism implies $\Res(\E)=\Res(E)$. Since $\Res$ is injective on $E_2(N)$, this shows that $E=\E$, proving (1).  The map described in (2) is the composition of the isomorphisms in \eqref{eq:X subquo C}. 

The homomorphism $\lambda$ defined in (3) is the composition of the isomorphism defined in (2) with the map $\lambda': X \to \Q_p/\Z_p$ defined by $\lambda'(ae+bg)=ba_1(g) \bmod{\Z_p}$.  By (2), it is enough to show that $\lambda'$ generates the dual of $X$.  First note that $\lambda'$ is well defined since $M[1/p]=S[1/p] \oplus E[1/p]$ as $\Q$-vector spaces. The element of $\Ext^1_{\Z_p}(X,\Z_p)$ corresponding to $\lambda'$ is class of the extension
\[
0 \to \Z_p \to Z' \to X \to0
\]
where $Z'=\{(x,\alpha) \in X \times \Q_p \ | \ \lambda'(x) \equiv \alpha \bmod{\Z_p}\}$.  On the other hand, the proof of \cref{cor:subquo of C} shows that $\Ext^1_{\Z_p}(X,\Z_p)$ is a free $\bT/I$-module generated by the class of the extension
\[
0 \to \Z_p\to Z \to X \to 0
\]
where $Z =\displaystyle \frac{M/E \oplus \Z_p}{\dia{(g,-a_1(g)) \ : \ g \in S}}$. 
There is a map
\[
Z \to Z', \quad (f, \alpha) \mapsto (f , \alpha + ba_1(g))
\]
where $f=ae+bg \in M$ with $e \in E$, $g \in S$, $a,b \in \Q_p$, and $\alpha \in \Z_p$.  A simple computation shows that this map is an isomorphism of extensions. Hence the class of the extension given by $Z'$ generates $\Ext^1_{\Z_p}(X,\Z_p)$ and $\lambda'$ generates the dual of~$X$.
\end{proof}

\section{Complement: explicit bases of Eisenstein series and modular units}  We give explicit bases of $\E$ and $U$ as $\Zp$-modules.  For $d \mid N$ with $d>1$, let $h_d \in U$ be the scalar class of $\left( \frac{\eta(dz)}{\eta(z)}\right)^{12N}$, where $\eta(z)$ is the Dekekind eta function. 
Let $f_d=\cD(h_d) \in \E$, and note that
\[
f_d(z) = 12N (E_2(dz)-E_2(z)).
\]
where $E_2(z)$ is the (non-holomorphic) normalized Eisenstein series of weight $2$ and level $1$.
\begin{cor} \hfill 
\label{cor:lattices equal}
\begin{enumerate}
\item The set  $\{f_d \ : \ d>1, \ d \mid N\}$ is a basis for $E_2(N,\Zp)$ as a $\Zp$-module.
\item The period lattice $\E$ coincides with the $q$-expansion lattice $E_2(N,\Zp)$.
\item The set $\{h_d \ : \ d>1, \ d \mid N\}$ is a basis for $U$ as a $\Zp$-module.
\end{enumerate}
\end{cor}
\begin{proof}
Note that (2) has already been proven in \cref{cor: gens of C dual}(1); we give an alternate proof here.
Since each $f_d$ is in $\E$,  (1) implies that $E_2(N,\Zp) \subseteq \E$. Together with \cref{lem: period lattice in q lattice}, this inclusion implies (2).
Since $f_d=\cD(h_d)$ and $\cD$ is an isomorphism, (1) and (2) imply (3).

It remains to prove (1).  Let $E' \subseteq E_2(N,\Zp)$ be the $\Zp$-submodule generated by $\{f_d \ | \ d>1, \ d \mid N\}$.   Since the set $\{f_d \ | \ d>1, \ d \mid N\}$ is a $\Q$-basis for $E_2(N,\Q)$ (see \cite[Theorem 4.6.2, pg.~133]{DS2005}, for example),  the index of $E'$ in $E_2(N,\Zp)$ is finite.  To show this index is $1$, it is enough to show that the restriction map
\begin{equation}
\label{eq:Hom E}
\Hom_{\Zp}(E_2(N,\Zp),\Zp) \to \Hom_{\Zp}(E',\Zp)
\end{equation}
is surjective.

For any divisor $d \mid N$, define a functional $l_d: E_2(N,\Zp)\to \Zp$ by
\begin{equation}
\label{eq:l_d}
l_d(f) = \sum_{t \mid d} \mu(d/t) \sigma_{d/t} a_t(f).
\end{equation}
where $\mu$ is the M\"obius function and $\sigma$ is the divisor-sum function.  
In terms of the duality pairing of Lemma \ref{lem:duality},  $l_d$ corresponds to the Hecke operator 
\[
t_d:=  \prod_{\{\ell \mid N : \ell \text{ prime}\}}(U_\ell-\ell-1),
\]
in that $l_d(f) = a_1 \left(t_d f \right)$.
Now let $d, s \mid N$ with $d,s>1$.  An elementary computation,  either directly using \eqref{eq:l_d} or by computing the action of $t_d$ on $f_s$, shows
\[
l_d(f_s) = \begin{cases}
-12Nd & s=d \\
0 & s \ne d.
\end{cases} 
\]
Since $-12Nd \in \Zp^\times$, the elements $\ell_d$ generate $\Hom_{\Zp}(E',\Zp)$ and the map \eqref{eq:Hom E} is surjective.
\end{proof}

\cref{cor:lattices equal}(3) was first proven by Takagi \cite[Theorem 4.1]{takagi1997} using Kubert--Lang theory. 

\bibliographystyle{alpha}
\bibliography{oct2018}

 \end{document}